\documentclass[11pt,a4paper]{article}
\usepackage{a4wide,amsfonts,amsmath,latexsym,amssymb,euscript,eufrak,graphicx,units,mathrsfs}
\usepackage{amssymb}
\usepackage{stmaryrd}
\usepackage{color}
\usepackage{float}
\newfloat{figure}{H}{lof}
\floatname{figure}{\figurename}

\DeclareMathAlphabet{\eufrak}{U}{}{}{}  % Euler fraktur math
\SetMathAlphabet\eufrak{normal}{U}{euf}{m}{n}
\SetMathAlphabet\eufrak{bold}{U}{euf}{b}{n}

\usepackage{amsmath, amsthm, amsfonts, amssymb}

\numberwithin{equation}{section}

\def\real{{\mathord{{\rm I\kern-2.8pt R}}}}        % Fake blackboard bold R.
\def\inte{{\mathord{{\rm I\kern-2.8pt N}}}}
\def\PP{{\mathord{{\rm I\kern-2.8pt P}}}}

\def\real{{\mathord{\mathbb R}}}

\def\inte{{\mathord{\mathbb N}}}

\def\R{\right}

\def\R{\mathbb{R}}

\def\E{\mathop{\hbox{\rm I\kern-0.20em E}}\nolimits}

\def\real{\mathbb{R}}

\newtheorem{prop}{Proposition}[section]

\newtheorem{lemma}[prop]{Lemma}

\newtheorem{theorem}[prop]{Theorem}

\newtheorem{remark}[prop]{Remark}

\allowdisplaybreaks
\begin{document}

\renewcommand{\thefootnote}{\fnsymbol{footnote}}

\begin{center}
{\large{\bf Berry-Ess\'een bounds for the least squares estimator
for discretely observed
 fractional Ornstein-Uhlenbeck processes    }}\\~\\
Khalifa Es-Sebaiy\footnote{{  National School of Applied Sciences -
Marrakesh, Cadi Ayyad University, Morocco.} E-mail: {\tt
k.essebaiy@uca.ma} }
 \\
{\it {  Cadi Ayyad University}}  \\~\\
\end{center}
{\small \noindent {\bf Abstract:} Let $\theta>0$. We consider a
one-dimensional fractional Ornstein-Uhlenbeck  process defined as
$dX_t= -\theta\ X_t dt+dB_t,\ t\geq0,$ where $B$ is a fractional
Brownian motion of Hurst parameter $H\in(\frac{1}{2},1)$. We are
interested in the problem of estimating the unknown parameter
$\theta$. For that purpose, we dispose of a discretized trajectory,
observed at $n$ equidistant times $t_i=i\Delta_{n}, i=0,\ldots,n$,
and $T_n=n\Delta_{n}$  denotes the length of the `observation
window'. We assume that $\Delta_{n} \rightarrow 0$ and
$T_n\rightarrow \infty$ as $n\rightarrow \infty$. As an estimator of
$\theta$ we choose the least  squares estimator (LSE)
$\widehat{\theta}_{n }$. The consistency of this estimator is
established. Explicit bounds for the Kolmogorov distance, in the
case when $H\in(\frac{1}{2},\frac{3}{4})$, in the central limit
theorem for the LSE $\widehat{\theta}_{n }$ are obtained. These results hold without any kind of ergodicity on the process $X$.\\

{\small \noindent {\bf Key words:} fractional Ornstein-Uhlenbeck
processes, discrete-time observation, least squares estimator,
Kolmogorov distance, central limit theorem, Malliavin calculus.

\section{Introduction}\label{intro}

In this paper we consider a fractional Ornstein-Uhlenbeck process
$X=(X_t,t\geq0)$. That is, it solves the  linear stochastic
differential equation
\begin{equation}\label{eqOU}
X_0=x_0;\quad dX_t= -\theta\ X_t dt+dB_t,\quad t\geq0,
\end{equation}
where $x_0\in\R$,  $B=(B_t,t\geq0)$ is a fractional Brownian motion
with Hurst parameter $H\in(\frac{1}{2},1)$ and  $\theta>0$ is
an unknown parameter.\\
Assume that the process   $X$ is observed equidistantly in time with
the step size $\Delta_n$: $t_i=i\Delta_{n}, i=0,\ldots,n$, and
$T_n=n\Delta_{n}$ denotes the length of the `observation window'.
The purpose of this paper is to study the least squares estimator
(LSE) $\widehat{\theta}_{n }$ of $\theta$ based on the sampling data
$X_{t_i}, i=0,\ldots,n$.
\\The LSE $\widehat{\theta}_{n }$  is obtained as
follows: $\widehat{\theta}_{n}$ minimizes
\[
\theta\mapsto
\sum_{i=1}^{n}\int_{t_{i-1}}^{t_i}\left|\dot{X}_{t}+\theta
X_{t_{i-1}}\right|^2dt,
\]where $t_i=i\Delta_{n}, i=0,\ldots,n$. Thus
$\widehat{\theta}_n$ is given by
\begin{equation}\label{estimator}
\widehat{\theta}_{n}=-\frac{\sum_{i=1}^{n}\int_{t_{i-1}}^{t_i}
X_{t_{i-1}}\delta X_t}{\Delta_{n}\sum_{i=1}^{n} X_{t_{i-1}}^2}.
\end{equation}
Also, by using (\ref{eqOU}), we arrive to the following formula:
\begin{eqnarray}\label{representation estimator-theta}
\widehat{\theta}_{n}-\theta&=&-\frac{\sum_{i=1}^{n}
U_i}{\Delta_{n}\sum_{i=1}^{n} X_{t_{i-1}}^2}
\end{eqnarray}where
{  \begin{eqnarray*}\label{expression of U} U_i=\theta
X_{t_{i-1}}\int_{t_{i-1}}^{t_{i}}(X_{t_{i-1}}-X_s)ds+\int_{t_{i-1}}^{t_i}
X_{t_{i-1}} \delta B_t,\qquad
 i=1,\ldots,n.
\end{eqnarray*}}

The parametric estimation problems for fractional diffusion
processes based on continuous-time observations have been studied
e.g. in \cite{KL, TV, Rao05, Rao2010} via maximum likelihood method.
Recently, the parametric estimation of the continuously observed
fractional Ornstein-Uhlenbeck process defined in (\ref{eqOU}) is
 studied by using the least
 squares estimator (LSE) defined by
 \begin{eqnarray*}\widetilde{\theta}_T=-\frac{\int_0^TX_t\delta X_t}{\int_0^TX_t^2dt}
\end{eqnarray*}In the case  $\theta>0$,  Hu and Nualart \cite{HN}
proved that the LSE $\widetilde{\theta}_T$ of $\theta$ is strongly
consistent and asymptotically normal. In addition, they also proved
that the following estimator
 \begin{eqnarray*}\overline{\theta}_T=\left(\frac{1}{H\Gamma(H)T}{\int_0^TX_t^2dt}\right)^{-\frac{1}{2H}}
\end{eqnarray*} is strongly consistent and asymptotically normal. In the case
$\theta<0$, Belfadli et al. \cite{BEO} established that the LSE
$\widetilde{\theta}_T$ of $\theta$ is strongly consistent and
asymptotically Cauchy.

From a practical point of view, in parametric inference, it is more
realistic and interesting to consider asymptotic estimation for
fractional diffusion processes based on discrete observations.\\
There exists a rich literature on the parameter estimation problem
for diffusion processes driven by Brownian motions  based on
discrete observations, , see \cite{Rao88} and \cite{Rao2010} for
more details about this point. For our  fractional
Ornstein-Uhlenbeck process (\ref{eqOU}), Hu and Song \cite{HS},
motivated by the estimator $\overline{\theta}_T$, proved that the
following estimator
\begin{eqnarray*}\underline{\theta}(n)=\left(\frac{1}{nH\Gamma(H)}\sum_{i=1}^{n}
X_{t_{i}}^2\right)^{-\frac{1}{2H}}
\end{eqnarray*} is strongly consistent, and they provided a Berry-Esseen type theorem for  $\underline{\theta}(n)$.
In this paper, we focus our discussion on the LSE case.

 In general, the study of the
asymptotic distribution of any estimator is not very useful for
practical purposes unless the rate of convergence of its
distribution is known. The rate of convergence of the distribution
of LSE for some diffusion processes driven by Brownian motions based
on discrete time data was studied e.g. in \cite{MR}. To the best of
our knowledge there is no study of this problem for the distribution
of the LSE of the unknown drift parameter in equation (\ref{eqOU}).
Our goal in the present paper is to investigate the consistency and
the rate of convergence to normality of the LSE $\widehat{\theta}_n$
defined in (\ref{estimator}).

Recall that, if $Y$, $Z$ are two real-valued random variables, then
the Kolmogorov distance between the law of  $Y$ and the law of $Z$
is given by
\begin{eqnarray*}d_{\mbox{\tiny{Kol}}}(Y,Z)=\sup_{-\infty<z<\infty}|P(Y\leq z)-P(Z\leq z)|.
\end{eqnarray*}
Let us now describe the results we prove in this work. In Theorem
\ref{consistency} we show that the consistency of
$\widehat{\theta}_n$ as $\Delta_{n} \rightarrow 0$ and
$n\Delta_{n}\rightarrow \infty$   holds true if
$H\in(\frac{1}{2},1)$. When $H\in(\frac{1}{2},\frac{3}{4})$ we use
the Malliavin calculus, the so-called Stein's method on Wiener chaos
introduced by \cite{Nou-Pecc} and the technical Lemmas \ref{lemma
quotient} and \ref{lemma technical} proved respectively by \cite{MP}
and \cite{BGS}, to derive Berry-Ess\'een-type bounds in the
Kolmogorov distance for the LSE $\widehat{\theta}_n$ (Theorems
\ref{main result} and \ref{main result second}).

We proceed as follows. In Section 2 we give the basic tools of
Malliavin calculus for the fractional Brownian motion needed
throughout the paper. Section 3 contains our main results,
concerning the consistency and the rate of convergence of
$\widehat{\theta}_n$.

\section{Preliminaries}
In this section we describe some basic facts on the  stochastic
calculus with respect to a
fractional Brownian motion. For more complete presentation on the subject, see   \cite{nualart-book} and \cite{AN}.\\
The  fractional Brownian motion $(B_t, t\geq0)$ with Hurst parameter
$H\in(0,1)$, is defined as a centered Gaussian process starting from
zero with covariance
\[R_H(t,s)=E(B_tB_s)=\frac{1}{2}\left(t^{2H}+s^{2H}-|t-s|^{2H}\right).\]
We assume that $B$ is defined on a complete probability space
$(\Omega, \mathcal{F}, P)$ such that $\mathcal{F}$ is the
sigma-field generated by $B$. By  Kolmogorov's continuity criterion
and the fact  $$E\left(B_t-B_s\right)^2=|s-t|^{2H};\ s,\ t\geq~0,$$
we deduce that $B$ {  admits a version which} has H\"older
continuous paths of any order $\gamma<H$.

 Fix a time interval $[0, T]$. We denote by $\cal{H}$ the canonical Hilbert space associated to the  fractional
Brownian motion $B$. That is, $\cal{H}$ is the closure of the linear
span $\mathcal{E}$ generated by the indicator functions $
1_{[0,t]},\ t\in[0,T] $ with respect to the scalar product
\[\langle1_{[0,t]},1_{[0,s]}\rangle=R_H(t,s).\]
The application $\varphi\in\mathcal{E}\longrightarrow B(\varphi)$ is an isometry from $\mathcal{E}$ to the Gaussian space generated by $B$
and it can be extended to $\cal{H}$.\\
If $H\in(\frac{1}{2},1)$ the elements of $\cal{H}$ may { not be }
functions but distributions of negative order (see \cite{PT}).
Therefore,
it is of interest to know significant subspaces of functions contained in it.\\
Let $|\cal{H}|$  be the set of measurable functions  $\varphi$   on
$[0, T]$ such that
\[\|\varphi\|_{|\cal{H}|}^2:=H(2H-1)\int_0^T\int_0^T|\varphi(u)||\varphi(v)||u-v|^{2H-2}dudv<\infty.\] Note that, if $\varphi,\ \psi\in|\cal{H}|$,
\[ E(B(\varphi)B(\psi))=H(2H-1)\int_0^T\int_0^T\varphi(u)\psi(v)|u-v|^{2H-2}dudv.\]
It follows actually from \cite{PT}  that the space $|\cal{H}|$ is a
Banach space for the norm $\|.\|_{|\cal{H}|}$ and it is included in
$\cal{H}$. In fact,
\begin{eqnarray}\label{inclusions} L^2([0, T]) \subset L^{\frac{1}{H}}([0, T])\subset|\cal{H}|\subset\cal{H}.
\end{eqnarray}

Let $\mathrm{C}_b^{\infty}(\R^n,\R)$
 be the class of infinitely
differentiable functions $f: \R^n \longrightarrow \R$ such that $f$
and all its partial derivatives are bounded. We denote by $\cal{S}$
the class of smooth cylindrical random variables F of the form
\begin{eqnarray}F = f(B(\varphi_1),...,B(\varphi_n)),\label{functional}\end{eqnarray} where $n\geq1$, $f\in \mathrm{C}_b^{\infty}(\R^n,\R)$
 and $\varphi_1,...,\varphi_n\in\cal{H}.$\\
 The derivative operator $D$ of a smooth  cylindrical random variable $F$ of the form (\ref{functional}) is
defined as the $\cal{H}$-valued random variable{
$$DF=\sum_{i=1}^{n}\frac{\partial f}{\partial
x_i}(B(\varphi_1),...,B(\varphi_n))\varphi_i.$$} In this way the
derivative $DF$ is an element of $L^2(\Omega ;\cal{H})$. We denote
by $D^{1,2}$ the closure of $\mathcal{S}$ with respect to the norm
defined by{
$$\|F\|_{1,2}^2=E(F^2)+E(\|DF\|^2_{{\cal{H}}}).$$ }
The divergence operator $\delta$ is the adjoint of the derivative
operator $D$. Concretely, a random variable $u\in
L^2(\Omega;\cal{H})$ belongs to the domain of the divergence
operator $Dom\delta$ if
\[E\left|\langle DF,u\rangle_{\cal{H}}\right|\leq c_u\|F\|_{L^2(\Omega)}\]for every $F\in \mathcal{S}$,
where $c_u$ is a constant which depends only on $u$. In this case
$\delta(u)$ is given by the duality relationship
\begin{eqnarray*}E(F\delta(u))=E\left<DF,u\right>_{\cal{H}}
\end{eqnarray*}
for any $F\in D^{1,2}$. We will make use of the notation
$$\delta(u)=\int_0^Tu_s\delta B_s,\quad u\in Dom\delta.$$
In particular, for $h\in\cal{H}$, $B(h)=\delta(h)=\int_0^Th_s\delta
B_s.$\\
Assume that $H\in(\frac{1}{2},1)$. If $u\in D^{1,2}(|{\cal{H}}|)$,
$u$ belongs to $Dom\delta$ and we have (see \cite[Page
292]{nualart-book})
\begin{eqnarray*}E(|\delta(u)|^2)\leq
c_H\left(\|E(u)\|_{|{\cal{H}}|}^2+E\left(\|Du\|_{|{\cal{H}}|\otimes|{\cal{H}}|}^2
\right)\right),
\end{eqnarray*}where the constant $c_H$ depends only on $H$.\\
As a consequence, applying (\ref{inclusions}) we obtain that
\begin{eqnarray}\label{majoration second moment of skorohod}
E(|\delta(u)|^2)\leq
c_H\left(\|E(u)\|_{L^{\frac{1}{H}}([0,
T])}^2+E\left(\|Du\|_{L^{\frac{1}{H}}([0, T]^2)}^2 \right)\right).
\end{eqnarray}

For every $n\geq1$, let ${\cal{H}}_n$ be the nth Wiener chaos of
$B$, that is, the closed linear subspace of $L^2(\Omega)$ generated
by the random variables $\{H_n(B(h)), h\in{{\cal{H}}},
\|h\|_{{\cal{H}}} = 1\}$ where $H_n$ is the nth Hermite polynomial.
The mapping ${I_n(h^{\otimes n})}=n!H_n(B(h))$ provides a linear
isometry between the symmetric tensor product ${\cal{H}}^{\odot n}$
(equipped with the modified norm $\|.\|_{{\cal{H}}^{\odot
n}}=\frac{1}{\sqrt{n!}}\|.\|_{{\cal{H}}^{\otimes n}}$) and
${\cal{H}}_n$. For every  $f,g\in{{\cal{H}}}^{\odot n}$ the
following   product formula holds
\[E\left(I_n(f)I_n(g)\right)=n!\langle f,g\rangle_{{\cal{H}}^{\otimes n}}.\]
 On the other hand, it is well-known   that $L^2(\Omega)$ can be decomposed into the infinite orthogonal sum of the spaces ${\cal{H}}_n$. That is, any square
integrable random variable $F\in L^2(\Omega)$  admits the following
chaotic expansion
\[F=E(F)+\sum_{n=1}^{\infty}I_n(f_n),\]
where the $f_n \in{{\cal{H}}}^{\odot n}$  are uniquely determined by
$F$.

We will make use of the following theorem proved in
\cite{Nou-Pecc}.{
\begin{theorem}[Nourdin-Peccati] \label{nourdin-peccati}Let  $F=I_q(f)$ with $q\geq2$ and $f\in{{\cal{H}}^{\odot
q}}$. Then,
\begin{eqnarray}d_{\mbox{\tiny{Kol}}}(F,N)
\leq\sqrt{E\left[\left(1-\frac{1}{q}\|DF\|_{{\cal{H}}}^2\right)^2\right]},
\end{eqnarray}
\end{theorem}where  $N\sim\mathcal{N}(0,1)$.}\\
Fix $T>0$. Let $f, g: [0,T]\longrightarrow\mathbb{R}$ { be} H\"older
continuous functions of orders $\alpha\in(0,1)$ and $\beta\in(0,1)$
respectively with $\alpha+\beta>1$. Young \cite{Young} proved that
the Riemann-Stieltjes integral (the so-called Young integral)
$\int_0^Tf_sdg_s$ exists. Moreover, if
$\alpha=\beta\in(\frac{1}{2},1)$ and $\phi:
\mathbb{R}^2\longrightarrow\mathbb{R}$ is a function of class
$\mathcal{C}^1$,
  the integrals $\int_0^.\frac{\partial\phi}{\partial f}(f_u,g_u)df_u$ and $\int_0^.\frac{\partial\phi}{\partial g}(f_u,g_u)dg_u$ exist in the Young sense and the following change of variables formula holds:
\begin{eqnarray}\label{change of variables formula}
\phi(f_t,g_t)=\phi(f_0,g_0)+\int_0^t\frac{\partial\phi}{\partial
f}(f_u,g_u)df_u+\int_0^t\frac{\partial\phi}{\partial
g}(f_u,g_u)dg_u,\quad 0\leq t\leq T.
\end{eqnarray}
As a consequence, if $H\in(\frac{1}{2},1)$ and $(u_t,\ t\in[0, T])$
{  is } a process with H\"older paths of order $\alpha\in(1-H,1)$,
the integral $\int_0^Tu_sdB_s$ is well-defined as {  a} Young
integral. Suppose moreover that for any  $ t\in[0,T]$, $u_t\in
D^{1,2}$, and
\[P\left(\int_0^T\int_0^T|D_su_t||t-s|^{2H-2}dsdt<\infty\right)=1.\]
Then, by \cite{AN}, $u\in Dom\delta$ and for every $ t\in[0,T]$,
\begin{eqnarray}\label{link}\int_0^tu_sdB_s=\int_0^tu_s\delta B_s+H(2H-1)\int_0^t\int_0^tD_su_r|s-r|^{2H-2}drds.
\end{eqnarray}
In particular, when $\varphi$ is a non-random  H\"older continuous
function of order $\alpha\in(1-H,1)$, we obtain
\begin{eqnarray}\label{non random}\int_0^T\varphi_sdB_s=\int_0^T\varphi_s\delta B_s=B(\varphi).\end{eqnarray}
\\In addition,  for all $\varphi,\ \psi\in|\cal{H}|$,
\begin{eqnarray}E\left(\int_0^T\varphi_sdB_s\int_0^T\psi_sdB_s\right)
=H(2H-1)\int_0^T\int_0^T\varphi(u)\psi(v)|u-v|^{2H-2}dudv.\end{eqnarray}

\section{Asymptotic behavior of the least squares estimator}
Throughout this paper we assume $H\in(\frac{1}{2},1)$ and
$\theta>0$. Let us consider the equation (\ref{eqOU}) driven by a
fractional Brownian motion $B$  with Hurst parameter $H$ and
$\theta$ is the unknown parameter to be estimated for discretely
observed $X$. The linear equation (\ref{eqOU}) has the following
explicit solution:
\begin{eqnarray}\label{explicit solution}X_t=e^{-\theta t}\left(x_0+\int_0^te^{\theta s}dB_s\right),\qquad t\geq0,
\end{eqnarray}
where the integral can be understood either in the Young sense, or in the Skorohod sense, see indeed (\ref{non random}).  \\
Let us introduce the following two processes related to $X$: for
$t\geq0$, \begin{eqnarray*}&&\xi_t=\int_0^te^{\theta s}dB_s;\\&&{
A_t=e^{-\theta t}\int_0^te^{\theta s}dB_s.}\end{eqnarray*}{  In
particular, we observe that
\begin{eqnarray}\label{second explicit solution}X_t=x_0e^{-\theta t}+A_t\qquad\mbox{ for
}t\geq0.
\end{eqnarray}}We shall use
the notation $a_n\trianglelefteqslant b_n$ to {  indicate} that
there exists a positive constant $c(x_0,\theta,H)$ (depending only
on $x_0,\theta$ and $H$) such that,
\[\sup_{n\geq1}|a_n|/|b_n|<c(x_0,\theta,H)<\infty.\]\\
We define the following sequence, which will be used throughout {
this paper},
 \begin{eqnarray}\label{alpha_n}\alpha_{n}=H(2H-1)\int_0^{T_n}  \int_0^t e^{-\theta
u}u^{2H-2}du dt,\quad n\geq0.\end{eqnarray}
  We shall be using the following
lemmas several times.
\begin{lemma}Let $H\in(\frac{1}{2},1)$, let $\theta>0$, and let $\alpha_{n}$ be the sequence defined
by (\ref{alpha_n}).
 Then
\begin{eqnarray}\label{link Young-Skorohod via alpha_n}
\int_0^{T_n}  X_sdB_s &=& \int_0^{T_n}  X_s\delta B_s+\alpha_{n},
\quad n\geq0,
\end{eqnarray} and
\begin{eqnarray}\label{limit alpha_{n}}
\lim_{n\rightarrow\infty}\frac{\alpha_{n} }{{T_n} }
 &=&\theta^{1-2H}H\Gamma(2H).
\end{eqnarray}In particular, $T_n\trianglelefteqslant \alpha_{n}$.
\end{lemma}
\begin{proof}By (\ref{link}), we have
\begin{eqnarray*}
\int_0^{T_n}  X_sdB_s &=& \int_0^{T_n} X_s\delta
B_s+H(2H-1)\int_0^{T_n} \int_0^{T_n}
D_sX_t|t-s|^{2H-2}dsdt\nonumber\\ &=& \int_0^{T_n} X_s\delta
B_s+H(2H-1)\int_0^{T_n}  \int_0^t
e^{-\theta(t-s)}(t-s)^{2H-2}dsdt\nonumber\\ &=& \int_0^{T_n}
X_s\delta
B_s+H(2H-1)\int_0^{T_n}  \int_0^t e^{-\theta u}u^{2H-2}dudt\nonumber\\
&=& \int_0^{T_n}  X_s\delta B_s+\alpha_{n}.
\end{eqnarray*}
On the other hand,
\begin{eqnarray*}
\lim_{n\rightarrow\infty}\frac{\alpha_{n} }{{T_n}
}&=&\lim_{n\rightarrow\infty}\frac{H(2H-1)}{{T_n} }\int_0^{T_n}
\int_0^t e^{-\theta
u}u^{2H-2}dudt\\&=&\lim_{n\rightarrow\infty}\frac{H(2H-1)}{{T_n}
}\int_0^{T_n}e^{-\theta u}u^{2H-2}(T_n-u)du\\
&=&  H(2H-1)\int_0^{\infty}
 e^{-\theta u}u^{2H-2}du\\
 &=&\theta^{1-2H}H\Gamma(2H).
\end{eqnarray*}
Thus, the proof is finished.
\end{proof}
\begin{lemma}\label{supX^2<infty}Assume $H\in(1/2,1)$ and $\theta>0$.
Then,   there exists a constant $c>0$, depending only on $x_0$,
$\theta$ and $H$, such that
\begin{eqnarray*}\sup_{t\geq0}E\left(X_t^2\right)<c<\infty.
\end{eqnarray*}
\end{lemma}
\begin{proof}For any $ t>0$, we have
\begin{eqnarray}E\left(\xi_t^2\right) &=&H(2H-1)\int_0^t\int_0^te^{\theta x}e^{\theta
y}|x-y|^{2H-2} dx dy\nonumber\\&=& 2H(2H-1)\int_0^tdye^{\theta
y}\int_0^ydxe^{\theta x}(y-x)^{2H-2}\nonumber\\&=&
2H(2H-1)\int_0^tdye^{2\theta y}\int_0^ydze^{ -\theta
z}z^{2H-2}\nonumber\\&=&{ 2} H(2H-1)\int_0^tdze^{ -\theta
z}z^{2H-2}\int_z^tdye^{2\theta y}\nonumber\\&=&
 H(2H-1)\int_0^tdze^{ -\theta z}z^{2H-2}\frac{e^{2\theta
 t}-e^{2\theta
 z}}{\theta}\nonumber\\&\leq&H(2H-1)\frac{\Gamma(2H-1)}{\theta^{2H}}e^{2\theta
 t}\nonumber\\&=&\frac{H\Gamma(2H)}{\theta^{2H}}e^{2\theta
 t}.\label{sup xi^2}
\end{eqnarray}
By combining   (\ref{explicit solution}) and (\ref{sup xi^2}), we
obtain that for any $t>0$
\begin{eqnarray*}E\left(X_t^2\right)
\leq2\left(x_0^2+\frac{H\Gamma(2H)}{\theta^{2H}}\right).
\end{eqnarray*}
This proves the claim.
\end{proof}
\subsection{Consistency of the  LSE}
 The next statement provides consistency  of the
LSE  $\widehat{\theta}_n$ of $\theta$.
\begin{theorem}\label{consistency}Assume $H\in(1/2,1)$ and $\theta>0$. Then, if $\Delta_{n} \rightarrow 0$ and $n\Delta_{n}\rightarrow
\infty$  as $n\rightarrow \infty$, we have
\begin{eqnarray}\widehat{\theta}_n \rightarrow\theta\quad \mbox{ in probability as  }n\rightarrow \infty.
\end{eqnarray}
\end{theorem}
\begin{proof} From (\ref{representation estimator-theta}), we can
write
\begin{eqnarray*}
\widehat{\theta}_{n}-\theta&=&-\frac{\frac{\theta}{\alpha_{n}}\sum_{i=1}^{n}
U_i}{\frac{\theta\Delta_{n}}{\alpha_n}\sum_{i=1}^{n} X_{t_{i-1}}^2}.
\end{eqnarray*}
Let $0<\rho<1$. We have
\begin{eqnarray*}
P\left(\left|\widehat{\theta}_{n}-\theta\right|>\rho\right)&=&P\left(\left|
\frac{\frac{\theta}{\alpha_{n}}\sum_{i=1}^{n}
U_i}{\frac{\theta\Delta_{n}}{\alpha_n}\sum_{i=1}^{n}
X_{t_{i-1}}^2}\right|>\rho\right)\\&\leq&P\left(\left|
 \frac{\theta}{\alpha_{n}}\sum_{i=1}^{n}
U_i \right|>\rho(1-\rho)\right)+P\left(\left|
\frac{\theta\Delta_{n}}{\alpha_n}\sum_{i=1}^{n}
X_{t_{i-1}}^2 -1\right|>\rho\right)\\
&:=&j_1(n)+j_2(n).
\end{eqnarray*}
We begin {  by} studying   the term $j_1(n)$. We have,
\begin{eqnarray*}j_1(n)&=&
P\left(\left|
 \frac{\theta}{\alpha_{n}}\sum_{i=1}^{n}
U_i \right|{ >}\rho(1-\rho)\right)\\&=& P\left(\left|
 \frac{\theta}{\alpha_{n}}\sum_{i=1}^{n}
[U_i-\int_{t_{i-1}}^{t_{i}}X_{t_{i-1}}{\delta}B_t]+\frac{\theta}{\alpha_{n}}
\sum_{i=1}^{n}\int_{t_{i-1}}^{t_{i}}(X_{t_{i-1}}-X_t){\delta}B_t\right.
\right.\\&&\left.\left.\qquad\qquad\qquad+\frac{\theta}{\alpha_{n}}\int_0^{T_n}
X_t{\delta}B_t \right|{ >}\rho(1-\rho)\right)
\\&\leq&P\left(
 \frac{\theta}{\alpha_{n}}\left|\sum_{i=1}^{n}
[U_i-\int_{t_{i-1}}^{t_{i}}X_{t_{i-1}}{\delta}B_t]\right|{
>}\frac{1}{3}\rho(1-\rho)\right)\\&&+ P\left(
 \frac{\theta}{\alpha_{n}}\left|\sum_{i=1}^{n}\int_{t_{i-1}}^{t_{i}}(X_{t_{i-1}}-X_t){\delta}B_t\right|{ >}\frac{1}{3}\rho(1-\rho)\right)
\\&&+P\left(
 \frac{\theta}{\alpha_{n}}\left|\int_0^{T_n}
X_t{\delta}B_t \right|{ >}\frac{1}{3}\rho(1-\rho)\right)
\\&:=&j_{1,1}(n)+j_{1,2}(n)+j_{1,3}(n).
\end{eqnarray*}
For the term $j_{1,1}(n)$, by using Lemma \ref{supX^2<infty} and the
fact that for every $t>0$
\begin{eqnarray}\label{X_ti-X_t}
 X_{t_{i-1}}-X_t&=&(e^{-\theta t_{i-1}}-e^{-\theta t })(x_0+\xi_{t_{i-1}})+e^{-\theta t
 }(\xi_{t_{i-1}}-\xi_{t }),
\end{eqnarray}
we obtain
\begin{eqnarray}
  \sum_{i=1}^{n}
E\left|U_i-\int_{t_{i-1}}^{t_{i}}X_{t_{i-1}}{\delta}B_t\right|&=&\sum_{i=1}^{n}
E\left|X_{t_{i-1}} \int_{t_{i-1}}^{{t_i}}\theta(X_{t_{i-1}}-X_t)dt
\right|\nonumber
 \\&\leq&\theta\sum_{i=1}^{n}
(E [X_{t_{i-1}}^2])^{1/2}
\int_{t_{i-1}}^{{t_i}}\left(E\left(\left[X_{t_{i-1}}-X_t\right]^2\right)\right)^{1/2}dt
\nonumber\\&\trianglelefteqslant& \sum_{i=1}^{n}
 \int_{t_{i-1}}^{{t_i}}\left(E\left([X_{t_{i-1}}-X_t]^2\right)\right)^{1/2}dt \label{estimation of X_ti-X_t}\\
&\trianglelefteqslant&  \sum_{i=1}^{n}
\int_{t_{i-1}}^{{t_i}}[e^{-\theta t_{i-1}}-e^{-\theta t
}] \left(E\left([x_0+\xi_{  t_{i-1} }]^2\right)\right)^{1/2}dt \nonumber\\
&&+   \sum_{i=1}^{n} \int_{t_{i-1}}^{{t_i}}e^{-\theta t}\left(
E\left( [\xi_t-\xi_{ t_{i-1} } ]^2\right)\right)^{1/2}dt
\nonumber\\
&\trianglelefteqslant&   \sum_{i=1}^{n} \int_{t_{i-1}}^{{t_i}}
[1-e^{-\theta( t-t_{i-1})
}]   dt  \nonumber\\
&&+  \sum_{i=1}^{n}  \int_{t_{i-1}}^{{t_i}}e^{-\theta
t}\left(E\left( [\xi_t-\xi_{ t_{i-1} } ]^2\right)\right)^{1/2}dt.
\nonumber
\end{eqnarray}
Making the change of variables $s= t-t_{i-1}$, we obtain
\begin{eqnarray*}
\sum_{i=1}^{n}  \int_{t_{i-1}}^{{t_i}}  [1-e^{-\theta( t-t_{i-1}) }]
 dt  &=&n
\int_{0}^{\Delta_{n}} [1-e^{-\theta s }]    ds  \\
&=&n\Delta_{n}^{2}   \frac{\int_{0}^{\Delta_{n}} [1-e^{-\theta s }]
 ds}{\Delta_{n}^{2}}  \\&\trianglelefteqslant&n \Delta_{n}^{2}
\end{eqnarray*} where the last estimate comes from the fact that  $ \int_{0}^{\Delta_{n}}
[1-e^{-\theta s }] ds/\Delta_{n}^{2} \rightarrow { \theta}/2$ as
$\Delta_{n}\rightarrow 0$ (by L'H\^opital's rule). On the other
hand, by the change of variables $u=\frac{x-t_{i-1}}{t-t_{i-1}}$,
$v=\frac{y-t_{i-1}}{t-t_{i-1}}$ and $s= t-t_{i-1} $
\begin{eqnarray*}&&\sum_{i=1}^{n}
\int_{t_{i-1}}^{{t_i}}e^{-\theta t}\left(E( [\xi_t-\xi_{ t_{i-1} }
]^2)\right)^{1/2}dt \\&=&\sqrt{H(2H-1)}
 \sum_{i=1}^n\int_{t_{i-1}}^{t_{i }}dt\ e^{
-\theta t}\left(\int_{t_{i-1}}^{t }dy e^{ \theta y}\int_{t_{i-1}}^{t
}dx e^{ \theta x}|x-y|^{2H-2}  \right)^{1/2}
\\&=&\sqrt{H(2H-1)}
 \sum_{i=1}^n\int_{t_{i-1}}^{t_{i }}dte^{
-\theta (t-t_{i-1})}\left(\int_{t_{i-1}}^{t }dye^{  \theta(y-
t_{i-1})}\int_{t_{i-1}}^{t }dxe^{
 \theta(x-t_{i-1} )}|x-y|^{2H-2}    \right)^{1/2}
 \\&=&\sqrt{H(2H-1)}
 \sum_{i=1}^n\int_{t_{i-1}}^{t_{i }}dt(t-t_{i-1})^{H}e^{
-\theta (t-t_{i-1})}\left(\int_{0}^{1 }dve^{
\theta(t-t_{i-1})v}\int_{0}^{1 }du e^{
 \theta(t-t_{i-1})u}|u-v|^{2H-2}   \right)^{1/2}
 \\&=&n\sqrt{H(2H-1)}
 \int_{0}^{\Delta_{n}} ds\ s^{H}e^{
-\theta  s}\left(\int_{0}^{1 }\int_{0}^{1 }e^{
 \theta\Delta_{n} s u}e^{  \theta\Delta_{n} s v}|u-v|^{2H-2} du dv\right)^{1/2}
 \\&\trianglelefteqslant&n\int_{0}^{\Delta_{n}} \ s^{H}ds
 \\&\trianglelefteqslant&n\Delta_{n}^{ H+1}.
\end{eqnarray*} Hence, we obtain that
\begin{eqnarray}
  \sum_{i=1}^{n}
E\left|U_i-\int_{t_{i-1}}^{t_{i}}X_{t_{i-1}}{\delta}B_t\right|&\trianglelefteqslant&n\left(
\Delta_{n}^{2}+ \Delta_{n}^{ H+1}
\right)\nonumber\\&\trianglelefteqslant& n
 \Delta_{n}^{ H+1}\label{majoration of j_{1,1}(n)^0},
\end{eqnarray}
which leads to
\begin{eqnarray*}
 \frac{\theta}{\alpha_{n}} \sum_{i=1}^{n}
E\left|U_i-\int_{t_{i-1}}^{t_{i}}X_{t_{i-1}}{\delta}B_t\right|\trianglelefteqslant
\frac{n \Delta_{n}^{ H+1}}{T_n}=\Delta_{n}^{H}.
\end{eqnarray*}
Consequently,
\begin{eqnarray}\label{majoration of j_{1,1}(n)}
j_{1,1}(n)=P\left(
 \frac{\theta}{\alpha_{n}}\left|\sum_{i=1}^{n}
[U_i-\int_{t_{i-1}}^{t_{i}}X_{t_{i-1}}{\delta}B_t]\right|{
>}\frac{1}{3}\rho(1-\rho)\right)
 \trianglelefteqslant \frac{\Delta_{n}^{H}}{\rho(1-\rho)}.
\end{eqnarray}
For the term $j_{1,2}(n)$, from (\ref{X_ti-X_t}), we have
\begin{eqnarray*}
 E\left|\sum_{i=1}^{n}\int_{t_{i-1}}^{t_{i}}(X_{t_{i-1}}-X_t){\delta}B_t\right|&\leq&
 E\left|\sum_{i=1}^{n}\int_{t_{i-1}}^{t_{i}}\left((e^{-\theta t_{i-1}}-e^{-\theta t })(x_0+\xi_{t_{i-1}})\right)
 {\delta}B_t\right|
 \\&&+E\left|\sum_{i=1}^{n}\int_{t_{i-1}}^{t_{i}}\left(e^{-\theta t
 }(\xi_{t_{i-1}}-\xi_{t })\right){\delta}B_t\right|.
\end{eqnarray*}
Using the inequality (\ref{majoration second moment of skorohod}){ ,
$E\xi_t=0$ and $D_s\xi_t=e^{\theta s}1_{[0,t]}(s)$}, we can write
\begin{eqnarray*}
&&E\left|\sum_{i=1}^{n}\int_{t_{i-1}}^{t_{i}}\left((e^{-\theta
t_{i-1}}-e^{-\theta t
})(x_0+\xi_{t_{i-1}})\right){\delta}B_t\right|\\&=&E\left|\int_{0}^{T_n}\sum_{i=1}^{n}\left((e^{-\theta
t_{i-1}}-e^{-\theta t
})(x_0+\xi_{t_{i-1}})\right)1_{(t_{i-1},t_i]}(t){\delta}B_t\right|
\\&{ \leq}&{ \left(E\left(\left|\int_{0}^{T_n}\sum_{i=1}^{n}\left((e^{-\theta
t_{i-1}}-e^{-\theta t
})(x_0+\xi_{t_{i-1}})\right)1_{(t_{i-1},t_i]}(t){\delta}B_t\right|^2\right)\right)^{1/2}}\\&\leq&c_H\left[
{  \left(\int_{0}^{T_n}\left|\sum_{i=1}^{n}x_0(e^{-\theta
t_{i-1}}-e^{-\theta t })1_{(t_{i-1},t_i]}(t)
\right|^{1/H}dt\right)^{H}}\right.\\&&+\left.\left(\int_{0}^{T_n}\int_{0}^{T_n}\left|\sum_{i=1}^{n}(e^{-\theta
t_{i-1}}-e^{-\theta t })D_s\xi_{t_{i-1}}1_{(t_{i-1},t_i]}(t)
\right|^{1/H}dsdt\right)^{H}\right]\\&=&c_H\left[ {
|x_0|\left(\int_{0}^{T_n}\sum_{i=1}^{n}\left|e^{-\theta
t_{i-1}}-e^{-\theta t }\right|^{1/H}1_{(t_{i-1},t_i]}(t)
dt\right)^{H}}\right.\\&&+\left.
 \left(\int_{0}^{T_n}\int_{0}^{T_n}\sum_{i=1}^{n}\left|(e^{-\theta
t_{i-1}}-e^{-\theta t })e^{\theta s}\right|^{1/H}1_{[0,t_{i-1}
]}(s)1_{(t_{i-1},t_i]}(t) dsdt\right)^{H}\right]\\&=&c_H\left[ {
\left( |x_0|\sum_{i=1}^{n}\int_{t_{i-1}}^{t_{i}}e^{-\theta
t_{i-1}}\left|1-e^{-\theta (t-t_{i-1}) }\right|^{1/H}
dt\right)^{H}}\right.\\&&+\left.
 \left(\sum_{i=1}^{n}\int_{t_{i-1}}^{t_{i}}dt \left[(1 -e^{-\theta
(t-t_{i-1}) } \right]^{1/H}\int_{0}^{t_{i-1}}ds e^{-\theta
(t_{i-1}-s)/H} \right)^{H}\right]\\&=&c_H\left[{
|x_0|\left(\sum_{i=1}^{n}\int_{0}^{\Delta_n}e^{-\theta
t_{i-1}}\left|1-e^{-\theta s }\right|^{1/H}
ds\right)^{H}}\right.\\&&+\left.\left(\sum_{i=1}^{n}
 \int_{0}^{\Delta_{n}} dv \left[ 1 -e^{-\theta
v } \right]^{1/H}\int_{0}^{t_{i-1}}du e^{-\theta u/H}
\right)^{H}\right]\\&\leq&c_H{
(|x_0|+1)}n^H\left(\int_{0}^{\Delta_n}\left|1-e^{-\theta s
}\right|^{1/H} ds\right)^{H}\\& \trianglelefteqslant &
n^H\Delta_{n}^{H+1}
\end{eqnarray*}because $ \int_{0}^{\Delta_{n}}  \left[ 1 -e^{-\theta
v } \right]^{1/H}dv/\Delta_{n}^{1+1/H}
\rightarrow\theta^{1/H}/(1+1/H)$ as $\Delta_{n}\rightarrow 0$ (by
L'H\^opital's rule). Similarly, we have
\begin{eqnarray*}
E\left|\sum_{i=1}^{n}\int_{t_{i-1}}^{t_{i}}e^{-\theta t
 }(\xi_{t_{i-1}}-\xi_{t }){\delta}B_t\right|&=&E\left|\int_{0}^{T_n}\sum_{i=1}^{n}e^{-\theta t
 }(\xi_{t_{i-1}}-\xi_{t })1_{(t_{i-1},t_i]}(t){\delta}B_t\right|\\&\leq&c_H
 \left(\int_{0}^{T_n}\int_{0}^{T_n}\left|\sum_{i=1}^{n}e^{-\theta t
 }D_s(\xi_{t_{i-1}}-\xi_{t })1_{(t_{i-1},t_i]}(t)
\right|^{1/H}dsdt\right)^{H}
\\&=&c_H
 \left(\int_{0}^{T_n}\int_{0}^{T_n}\sum_{i=1}^{n}\left|e^{-\theta t
 }D_s(\xi_{t_{i-1}}-\xi_{t })
\right|^{1/H}1_{(t_{i-1},t_i]}(t)dsdt\right)^{H}
\\&=&c_H
 \left(\int_{0}^{T_n}\int_{0}^{T_n}\sum_{i=1}^{n}\left|e^{-\theta t
 }e^{\theta s}
\right|^{1/H}1_{[{t_{i-1}} ,{t
}]}(s)1_{(t_{i-1},t_i]}(t)dsdt\right)^{H}\\&=&c_H
 \left(\sum_{i=1}^{n}\int_{t_{i-1}}^{t_i}dt\int_{t_{i-1}}^{t}ds e^{-\theta
 (t-s)/H}\right)^{H}\\&=&c_H
 \left(\sum_{i=1}^{n}\int_{t_{i-1}}^{t_i}dt\int_{0}^{t-t_{i-1}}du e^{-\theta
  u/H}\right)^{H}\\&=&c_Hn^H
 \left(\int_{0}^{\Delta_{n}} dv\int_{0}^{v}du e^{-\theta
  u/H}\right)^{H}\\&=&c_Hn^H
 \left(\int_{0}^{\Delta_{n}}  \frac{[1- e^{-\theta
   v/H}]}{\theta/H}dv\right)^{H}
   \\&\trianglelefteqslant&
  n^{H}\Delta_{n}^{2H}.
\end{eqnarray*}
Thus,
\begin{eqnarray}
E\left|\sum_{i=1}^{n}\int_{t_{i-1}}^{t_{i}}(X_{t_{i-1}}-X_t){\delta}B_t\right|&
\trianglelefteqslant&
 n^H\left(\Delta_{n}^{H+1}+\Delta_{n}^{2H}\right)\nonumber
 \\&\trianglelefteqslant&n^H \Delta_{n}^{{ {2H}}}.\label{majoration of j_{1,2}(n)^0}
\end{eqnarray}
Therefore,
\begin{eqnarray*}
\frac{\theta}{\alpha_{n}}E\left|\sum_{i=1}^{n}\int_{t_{i-1}}^{t_{i}}(X_{t_{i-1}}-X_t){\delta}B_t\right|
&\trianglelefteqslant& \frac{\Delta_{n}^{{ {2H-1}}}}{n^{1-H}}.
\end{eqnarray*}
As consequence, \begin{eqnarray} \label{majoration of
j_{1,2}(n)}j_{1,2}(n)=P\left(
 \frac{\theta}{\alpha_{n}}\left|\sum_{i=1}^{n}\int_{t_{i-1}}^{t_{i}}(X_{t_{i-1}}-X_t){\delta}B_t\right|
 >\frac{1}{3}\rho(1-\rho)\right)
&\trianglelefteqslant& \frac{\Delta_{n}^{{
{2H-1}}}}{n^{1-H}\rho(1-\rho)}.
\end{eqnarray}
For the term $j_{1,3}(n)$, by setting
\begin{eqnarray}\label{F_{T_n}}F_{T_n}=\frac{1}{\sqrt{T_n}}\int_0^{T_n}
{ A_t}{\delta}B_t
\end{eqnarray} we have
\begin{eqnarray}j_{1,3}(n)&=&P\left(
 \frac{\theta}{\alpha_{n}}\left|\int_0^{T_n}
X_t{\delta}B_t \right|>\frac{1}{3}\rho(1-\rho)\right)\nonumber\\
&\leq& { \left[\frac{ {3\theta}}{\alpha_{n}[\rho(1-\rho)]}
\right]^2E\left(\left|\int_0^{T_n} X_t{\delta}B_t
\right|^2\right)\nonumber}\\&\leq& { 2\left[\frac{
{3\theta}}{\alpha_{n}[\rho(1-\rho)]}
\right]^2\left[E\left(\left|\int_0^{T_n} x_0e^{-\theta t}{\delta}B_t
\right|^2\right)+E\left(\left|\int_0^{T_n} A_t{\delta}B_t
\right|^2\right)\right].}\label{calcul majoration of j_{1,3}(n)}
\end{eqnarray}
{  Recall that the following convergence holds (see \cite{HN} for
further details)
\begin{eqnarray}\label{limit of F_T^2}
E(F_{T_n}^2)\rightarrow A(\theta,H),\ \mbox{ as }
n\rightarrow\infty,
\end{eqnarray}where $$ A(\theta,H)=\theta^{1-4H}\left(H^2(4H-1)
\left[\Gamma(2H)^2+\frac{\Gamma(2H)\Gamma(3-4H)\Gamma(4H-1)}{\Gamma(2-2H)}\right]\right).$$
On the other hand,
\begin{eqnarray}E\left(\left|\int_0^{T_n} e^{-\theta t}{\delta}B_t
\right|^2\right)&=&H(2H-1)\int_0^{T_n}\int_0^{T_n}e^{-\theta
t}e^{-\theta
s}|t-s|^{2H-2}dsdt\nonumber\\&=&2H(2H-1)\int_0^{T_n}\int_0^{t}e^{-\theta
t}e^{-\theta
s}|t-s|^{2H-2}dsdt\nonumber\\&=&2H(2H-1)\int_0^{T_n}\int_0^{t}e^{-2\theta
t}e^{\theta
r}r^{2H-2}drdt\nonumber\\&=&2H(2H-1)\int_0^{T_n}dre^{\theta
r}r^{2H-2}\int_r^{T_n}dt e^{-2\theta t}\nonumber\\&\leq&
\frac{H(2H-1)}{\theta}\int_0^{\infty}e^{-\theta
r}r^{2H-2}dr\nonumber\\&=&
\frac{H(2H-1)}{\theta^{2H}}\Gamma(2H-1)<\infty.\label{double
integral finite}
\end{eqnarray}}
Combining (\ref{calcul majoration of j_{1,3}(n)}), (\ref{limit of
F_T^2}), (\ref{double integral finite}) and (\ref{limit alpha_{n}})
we have that
\begin{eqnarray}j_{1,3}(n)
&\trianglelefteqslant&\frac{1}{[\rho(1-\rho)]^{2}T_n}\label{majoration
of j_{1,3}(n)}.
\end{eqnarray}
 Finally, by combining (\ref{majoration of
j_{1,1}(n)}), (\ref{majoration of j_{1,2}(n)}) and (\ref{majoration
of j_{1,3}(n)}), we conclude that
\begin{eqnarray}\label{majoration of j_1(n)}
j_1(n)\trianglelefteqslant \frac{1}{[\rho(1-\rho)]^{2}}
\left(\Delta_{n}^{H}+\frac{\Delta_{n}^{{
{2H-1}}}}{n^{1-H}}+\frac{1}{n\Delta_n}\right).
\end{eqnarray}
Consequently, to achieve the proof of Theorem \ref{consistency}, it
remains to estimate the term $j_2(n)$. We have
\begin{eqnarray*}j_2(n)&=&P\left(\left|\frac{\theta}{\alpha_{n}}\Delta_{n}\sum_{i=1}^{n}X_{t_{i-1}}^2-1\right|>\rho\right)
\\&\leq&P\left(\left|\frac{\theta}{\alpha_{n}}\sum_{i=1}^{n}\int_{t_{i-1}}^{t_i}[X_{t_{i-1}}^2-X_t^2]dt \right|>\rho/2\right)
+P\left(\left|\frac{\theta}{\alpha_{n}}\int_{0}^{T_n} X_t^2
dt-1\right|>\rho/2\right)\\&:=&j_{2,1}(n)+j_{2,2}(n).
\end{eqnarray*}
We first estimate  $j_{2,1}(n)$. By using similar arguments as in
(\ref{estimation of X_ti-X_t}), we get
\begin{eqnarray*}
E\left|\frac{\theta}{\alpha_{n}}
\sum_{i=1}^{n}\int_{t_{i-1}}^{t_i}[X_{t_{i-1}}^2-X_t^2]dt
\right|&\leq & \frac{\theta}{\alpha_{n}}
\sum_{i=1}^{n}\int_{t_{i-1}}^{t_i}E\left| X_{t_{i-1}}^2-X_t^2
\right|dt\\&\leq &
\frac{2\theta\sup_{t\geq0}(E[X_t^2])^{1/2}}{\alpha_n}
\sum_{i=1}^{n}\int_{t_{i-1}}^{t_i}\left(E\left(\left|
X_{t_{i-1}}-X_t\right|^2\right)\right)^{1/2}dt\\
&\trianglelefteqslant&\frac{n\Delta_n^{2}+n\Delta_n^{H+1}}{T_n}
\\
&\trianglelefteqslant& \Delta_n^{H}
\end{eqnarray*}
 which implies that
\begin{eqnarray}\label{majoration of j_{2,1}(n)} j_{2,1}(n)&=& P\left(\left|\frac{\theta}{\alpha_{n}}
\sum_{i=1}^{n}\int_{t_{i-1}}^{t_i}[X_{t_{i-1}}^2-X_t^2]dt
\right|>\rho/2\right)\nonumber
\\&\trianglelefteqslant&\frac{\Delta_n^{H}}{\rho}.
\end{eqnarray}
We now study $j_{2,2}(n)$.
 Applying the change of variable formula
(\ref{change of variables formula}) leads to
\begin{eqnarray}\label{X^2 decomposition}
2\int_0^{T_n} X_tdB_t &=& X_{T_n}^2-x_0^2+2\theta  \int_0^{T_n}
X_t^2 dt.
\end{eqnarray}
Combining (\ref{link Young-Skorohod via alpha_n}) and (\ref{X^2
decomposition}) we obtain
\begin{eqnarray*}2\theta  \int_0^{T_n} X_t^2 dt-2\alpha_{n}=2\int_0^{T_n}
X_t\delta B_t-X_{T_n}^2+x_0^2.
\end{eqnarray*}
Hence
\begin{eqnarray}\label{int X_t^2-1}  \frac{\theta}{\alpha_{n}}  \int_0^{T_n} X_t^2 dt-1
=\frac{1}{\alpha_{n}}\int_0^{T_n} X_t\delta
B_t-\frac{1}{2\alpha_{n}}( X_{T_n}^2-x_0^2).
\end{eqnarray}
We deduce from  (\ref{int X_t^2-1}) and (\ref{limit of F_T^2})
together with (\ref{limit alpha_{n}}) and (\ref{double integral
finite}) that
\begin{eqnarray}
j_{2,2}(n)&=&P\left(\left|\frac{\theta}{\alpha_{n}}\int_{0}^{T_n}
X_t^2 dt-1\right|>\rho/2\right)\\&{ =}&
P\left(\left|\frac{1}{\alpha_{n}}\int_0^{T_n} X_t\delta
B_t-\frac{1}{2\alpha_{n}}(
X_{T_n}^2-x_0^2)\right|>\rho/2\right)\nonumber\\&{ \leq}&
P\left(\left|\frac{1}{\alpha_{n}}\int_0^{T_n} X_s\delta B_s
\right|>\rho/{ 4}\right)+P\left(\left|\frac{1}{2\alpha_{n}}(
X_{T_n}^2-x_0^2)\right|>\rho/{ 4}\right)\nonumber\\&\leq&{
2\left[\frac{ {4}}{\alpha_{n}\rho}
\right]^2\left[E\left(\left|\int_0^{T_n} x_0e^{-\theta t}{\delta}B_t
\right|^2\right)+E\left(\left|\int_0^{T_n} A_t{\delta}B_t
\right|^2\right)\right]}\nonumber\\&&{ +\left[\frac{
4}{\alpha_{n}\rho} \right] E|X_{T_n}^2-x_0^2|} \nonumber
\\&\trianglelefteqslant&  \left[\frac{
{T_n}}{\alpha_{n}} \right]^2\frac{{
1+E(F_{T_n}^2)}}{\rho^{2}T_n}+\left[\frac{ {T_n}}{\alpha_{n}}
\right] \frac{E|X_{T_n}^2-x_0^2|
}{\rho{T_n}}\nonumber\\&\trianglelefteqslant&
\left(\frac{1}{\rho^{2}{T_n}} +\frac{1}{\rho {T_n} }
\right)\nonumber\\&\trianglelefteqslant&
\frac{1}{\rho^{2}{T_n}}.\label{majoration of j_{2,2}(n)}
\end{eqnarray}
Therefore, from (\ref{majoration of j_{2,1}(n)}) and
(\ref{majoration of j_{2,2}(n)}), we obtain
\begin{eqnarray}\label{majoration of j_2(n)}
j_2(n) \trianglelefteqslant\frac{\Delta_n^H}{\rho}+
\frac{1}{\rho^{2}{T_n}}.
\end{eqnarray}
Finally, combining (\ref{majoration of j_1(n)}) and (\ref{majoration
of j_2(n)}), the proof of Theorem \ref{consistency} is done.
\end{proof}

\subsection{Rate of convergence of the  LSE}
This paragraph is devoted to derive Berry-Ess\'een-type bounds in
the Kolmogorov distance for the LSE $\widehat{\theta}_n$ of
$\theta$. We first recall the following technical lemmas.
\begin{lemma}[{\cite[Page 78]{MP}}]\label{lemma quotient}Let $f$ and $g$ be  two
real-valued random variables with $g\neq0$ $P$-a.s. Then, for any
$\delta>0$,
\begin{eqnarray*}d_{\mbox{\tiny{Kol}}}\left(\frac{f}{g},N\right)\leq
d_{\mbox{\tiny{Kol}}}\left(f,N\right)+P(|g-1|>\delta)+\delta.
\end{eqnarray*} where $N\sim\mathcal{N}(0,1)$.
\end{lemma}

\begin{lemma}[{\cite[Page
280]{BGS}}]\label{lemma technical}Let $Y$ and $Z$ be  two
real-valued random variables. Then, for any $\eta>0$,
\begin{eqnarray*}d_{\mbox{\tiny{Kol}}}\left(Y+Z,N\right)\leq
d_{\mbox{\tiny{Kol}}}\left( Z ,N
\right)+P(|Y|>\eta)+\frac{\eta}{\sqrt{2\pi}},
\end{eqnarray*}
where $N\sim\mathcal{N}(0,1)$.
\end{lemma}
Define
\[\lambda_n:=\frac{\alpha_n}{\theta T_n\sqrt{E(F_{T_n}^2)}},\quad n\geq1,\]
where $\alpha_n$ and $F_{T_n}$ are defined  by (\ref{alpha_n}) and (\ref{F_{T_n}}), respectively.\\
 The following result provides
explicit bounds for the Kolmogorov distance, in the case when
$H\in~(\frac{1}{2},\frac{3}{4})$, between the  law of
$\lambda_{n}\sqrt{T_n} (\widehat{\theta}_{n}-\theta)$ and the
standard normal law.
\begin{theorem}\label{main result}Let $(\delta,\eta)\in(0,1)^2$ and $H\in(\frac{1}{2},\frac{3}{4})$. If
$\Delta_n\rightarrow0$ and $n\Delta_n\rightarrow\infty$ then, for
some constant
 $c > 0$ depending uniquely on $x_0$, $\theta$ and $H$, we have: for any
 $n\geq1$,
\begin{eqnarray*} &&d_{\mbox{\tiny{Kol}}}\left(\lambda_n\sqrt{T_n}
(\widehat{\theta}_{n}-\theta),N\right)  \\&&\leq c\left({
\frac{1}{\eta \sqrt{n\Delta_{n}}}+}\frac{\sqrt{n}\Delta_{n}^{{{
2H}-\frac12}}}{\eta} +
(n\Delta_{n})^{4H-3}+\eta+\frac{\Delta_{n}^{H}}{\delta}+\frac{1}{n\Delta_{n}\delta^2}+\delta
\right),
\end{eqnarray*}
where   $N\sim\mathcal{N}(0,1)$.
\end{theorem}
\begin{proof}Fix $(\delta,\eta)\in(0,1)^2$. From (\ref{representation estimator-theta}) and Lemma
\ref{lemma quotient}, we obtain that
\begin{eqnarray*}&&d_{\mbox{\tiny{Kol}}}\left(\lambda_{n}\sqrt{T_n}
(\widehat{\theta}_{n}-\theta),N\right) \\&&\leq
d_{\mbox{\tiny{Kol}}}\left(
\frac{1}{\sqrt{T_n}\sqrt{E(F_{T_n}^2)}}\sum_{i=1}^{n} U_i
,N\right)+P\left(\left|\frac{\theta\Delta_n}{\alpha_n}\sum_{i=1}^{n}
X_{t_{i-1}}^2-1\right|>\delta\right)+\delta\\&&:=J_1(n)+J_2(n)+\delta.
\end{eqnarray*}
We first study the term $J_1(n)$. Using Lemma \ref{lemma technical},
we obtain
\begin{eqnarray*}J_1(n)&=&
d_{\mbox{\tiny{Kol}}}\left(
\frac{1}{\sqrt{T_n}\sqrt{E(F_{T_n}^2)}}\sum_{i=1}^{n} U_i ,N\right)
\\&=& d_{\mbox{\tiny{Kol}}}\left(
\frac{1}{\sqrt{T_n}\sqrt{E(F_{T_n}^2)}}\sum_{i=1}^{n}
[U_i-\int_{t_{i-1}}^{t_{i}}X_{t_{i-1}}{\delta}B_t]+\frac{1}{\sqrt{T_n}\sqrt{E(F_{T_n}^2)}}
\sum_{i=1}^{n}\int_{t_{i-1}}^{t_{i}}(X_{t_{i-1}}-X_t){\delta}B_t\right.
\\&&\left.\qquad\qquad\qquad{ +\frac{1}{\sqrt{T_n}\sqrt{E(F_{T_n}^2)}}\int_0^{T_n} x_0e^{-\theta t}{\delta}B_t}
+\frac{1}{\sqrt{T_n}\sqrt{E(F_{T_n}^2)}}\int_0^{T_n} {
A_t}{\delta}B_t,N\right)
\\&\leq&P\left(
\frac{1}{\sqrt{T_n}\sqrt{E(F_{T_n}^2)}}\left|\sum_{i=1}^{n}
[U_i-\int_{t_{i-1}}^{t_{i}}X_{t_{i-1}}{\delta}B_t]\right|\geq\frac{\eta}{{
3}}\right)\\&&+ P\left(
\frac{1}{\sqrt{T_n}\sqrt{E(F_{T_n}^2)}}\left|\sum_{i=1}^{n}\int_{t_{i-1}}^{t_{i}}(X_{t_{i-1}}-X_t){\delta}B_t\right|
\geq\frac{\eta}{{ 3}}\right)
\\&&{ +
P\left( \frac{1}{\sqrt{T_n}\sqrt{E(F_{T_n}^2)}}\left|\int_0^{T_n}
x_0e^{-\theta t}{\delta}B_t\right|\geq\frac{\eta}{3}\right)}
\\&&+d_{\mbox{\tiny{Kol}}}\left(\frac{1}{\sqrt{T_n}\sqrt{E(F_{T_n}^2)}}\int_0^{T_n}
{ A_t}{\delta}B_t,N\right)+\frac{\eta}{\sqrt{2\pi}}
\\&:=&J_{1,1}(n)+J_{1,2}(n)+J_{1,3}(n)+J_{1,4}(n)+\frac{\eta}{\sqrt{2\pi}}.
\end{eqnarray*}
 By using (\ref{majoration of j_{1,1}(n)^0}), (\ref{majoration of j_{1,2}(n)^0})  and (\ref{limit of F_T^2}) we deduce
 that{
\begin{eqnarray}\label{majoration of J_{1,1}(n)+J_{1,2}(n)}
J_{1,1}(n)+J_{1,2}(n)+J_{1,3}(n)&\trianglelefteqslant& \frac{3}{\eta
\sqrt{T_n}\sqrt{E(F_{T_n}^2)}}\left(n
 \Delta_{n}^{ H+1}+n^H \Delta_{n}^{{2H}}+1\right)\nonumber\\&\trianglelefteqslant& \frac{{1+} n
\Delta_{n}^{{{2H}}}}{\eta \sqrt{n\Delta_{n}}}\nonumber
 \\&\trianglelefteqslant&{ \frac{1}{\eta
\sqrt{n\Delta_{n}}}+}\frac{\sqrt{n}\Delta_{n}^{{{2H}-\frac12}}}{\eta}.
\end{eqnarray}}
 To achieve the estimation of $J_1(n)$ its remains to
estimate $J_{1,4}(n)$.  We have
\begin{eqnarray*}{ \frac{F_{T_n}}{\sqrt{E(F_{T_n}^2)}}=}\frac{1}{\sqrt{T_n}\sqrt{E(F_{T_n}^2)}}\int_0^{T_n}
{
A_t}{\delta}B_t=\frac{1}{2\sqrt{T_n}\sqrt{E(F_{T_n}^2)}}I_2\left(e^{-\theta
|t-s|}1_{[0,T_n]}^{\otimes2}(t,s)\right).
\end{eqnarray*}
Thus, by using  Theorem \ref{nourdin-peccati} and the fact that
$E\left(\|DF_{T_n}\|_{{\cal{H}}}^2\right)=2E\left(F_{T_n}^2\right)$,
we obtain
\begin{eqnarray*}J_{1,4}(n)=d_{\mbox{\tiny{Kol}}}\left(\frac{1}{\sqrt{T_n}\sqrt{E(F_{T_n}^2)}}\int_0^{T_n}
{
A_t}{\delta}B_t,N\right)&=&d_{\mbox{\tiny{Kol}}}\left(\frac{F_{T_n}}{\sqrt{E(F_{T_n}^2)}},N\right)\\&\leq&
\sqrt{E\left[\left(1-\frac{1}{2}\left\|\frac{DF_{T_n}}{\sqrt{E\left(F_{T_n}^2\right)}}\right\|_{{\cal{H}}}^2\right)^2\right]}\\&=&
\frac{1}{2E(F_{T_n}^2)}\sqrt{E\left[\left(\|DF_{T_n}\|_{{\cal{H}}}^2-E\|DF_{T_n}\|_{{\cal{H}}}^2\right)^2\right]}.
\end{eqnarray*}
Moreover, from \cite{HN}, we have
\[E\left[\left(\|DF_{T_n}\|_{{\cal{H}}}^2-E\|DF_{T_n}\|_{{\cal{H}}}^2\right)^2\right]
\trianglelefteqslant T^{8H-6}_n.\] Thus
\begin{eqnarray}\label{majoration of J_{1,4}(n)} J_{1,4}(n)
\trianglelefteqslant T_n^{4H-3}=(n\Delta_n)^{4H-3}.
\end{eqnarray}
Consequently, by combining (\ref{majoration of
J_{1,1}(n)+J_{1,2}(n)}) and (\ref{majoration of J_{1,4}(n)}) we have
\begin{eqnarray}\label{majoration of J_1(n)}
J_1(n)\trianglelefteqslant{ \frac{1}{\eta
\sqrt{n\Delta_{n}}}+}\frac{\sqrt{n}\Delta_{n}^{{{
2H}-\frac12}}}{\eta}+(n\Delta_n)^{4H-3} +\eta.
\end{eqnarray}
Finally, via the same arguments as in the estimation of  $j_2(n)$,
we obtain the following upper bound of
 $J_2(n)$:
 \begin{eqnarray}\label{majoration of J_2(n)}
J_2(n) \trianglelefteqslant\frac{\Delta_n^H}{\delta}+
\frac{1}{{n\Delta_n}\delta^{2}},
\end{eqnarray}
which completes the proof of Theorem \ref{main result}.
\end{proof}
Let us apply now Theorem \ref{main result} to the particular case ${
\eta=\sqrt{n\Delta_{n}^{\beta}}}$ and $\delta=\Delta_{n}^{\alpha}$,
where  $0<\alpha<H$ and $0<\beta<{ 4H-1}$ which ensure that {
$n\Delta_{n}^{1+2\alpha}\rightarrow\infty$,
$n\Delta_{n}^{\frac{1+\beta}{2}}\rightarrow\infty$ and
$n\Delta_{n}^{\beta }\rightarrow0$
 as $n\rightarrow\infty$.}
{ \begin{theorem}\label{main result second}Let $0<\alpha<H$ and
$1<\beta<4H-1$ such that
 $n\Delta_{n}^{1+2\alpha}\rightarrow\infty$, $n\Delta_{n}^{\frac{1+\beta}{2}}\rightarrow\infty$ and
$n\Delta_{n}^{\beta}\rightarrow0$
 as $n\rightarrow\infty$. If $H\in(\frac{1}{2},\frac{3}{4})$ then, for some constant $c> 0$  depending uniquely on
$x_0$, $\theta$ and $H$, we have: for any $n\geq1$,
\begin{eqnarray*} d_{Kol} \left(\lambda_{n}\sqrt{T_n}
(\widehat{\theta}_{n}-\theta),N\right)
&\leq&c\left(\frac{1}{n\Delta_{n}^{\frac{1+\beta}{2}}}+
\sqrt{\Delta_{n}^{4H-1-\beta}}+
(n\Delta_{n})^{4H-3}+\sqrt{n\Delta_{n}^{\beta}}+
\Delta_{n}^{H-\alpha}\right.
\\&&\qquad\left.+\frac{1}{n\Delta_{n}^{1+2\alpha}}+\Delta_{n}^{\alpha} \right).
\end{eqnarray*}
In particular, as $n\rightarrow\infty$
\begin{eqnarray*}\sqrt{T_n}
(\widehat{\theta}_{n}-\theta)\overset{\texttt{law}}{\longrightarrow}
\mathcal{ N}\left(0,\sigma_H^2\right)
\end{eqnarray*}where $\sigma_H^2=(4H-1)\theta\left(1+\frac{\Gamma(3-4H)\Gamma(4H-1)}{\Gamma(2-2H)\Gamma(2H)}\right)$.
\end{theorem}}
{ \begin{remark} As an example, assume that $\Delta_n= n^{-\gamma}$
with given $\gamma \in (\frac{1}{4H-1},\frac{1}{2H})$. Taking
$$0<\alpha<\frac{1-\gamma}{2\gamma}\mbox{ and }
\frac{1}{\gamma}<\beta<~4H-1$$ the conditions of Theorem \ref{main
result second} are satisfied.
\end{remark}}

\vspace{1cm} \noindent {\bf{Acknowledgement} }\\  The author would
like to thank the anonymous referee for his/her valuable suggestions
and remarks.

\bibliographystyle{amsplain}

\end{document}